\documentclass{amsart}
\usepackage{amsmath,amsfonts,amsthm}
\usepackage{amssymb,latexsym}
\usepackage{graphics}
\usepackage[colorlinks]{hyperref}

\theoremstyle{plain}
\theoremstyle{definition}
\newtheorem{theorem}{Theorem}[section]
\newtheorem{lemma}[theorem]{Lemma}

\newtheorem{corollary}[theorem]{Corollary}

\newtheorem{note}[theorem]{Note}

\newtheorem{remark}[theorem]{Remark}
\theoremstyle{remark}

\numberwithin{equation}{section}
\newcommand{\SP}{\: \: \: \: \:}

\title[Generalized shifts \& Compact operators]{Is there any nontrivial compact generalized shift operator on Hilbert spaces?}
\author[F. Ayatollah Zadeh Shirazi, F. Ebrahimifar]{Fatemah Ayatollah Zadeh Shirazi, Fatemeh Ebrahimifar}
\begin{document}
\begin{abstract}
In the following text for cardinal number $\tau>0$,  and self--map $\varphi:\tau\to\tau$ we show
the generalized shift operator $\sigma_\varphi(\ell^2(\tau))\subseteq\ell^2(\tau)$ (where
$\sigma_\varphi((x_\alpha)_{\alpha<\tau})=(x_{\varphi(\alpha)})_{\alpha<\tau}$ for
$(x_\alpha)_{\alpha<\tau}\in{\mathbb C}^\tau$) if and only if
$\varphi:\tau\to\tau$ is bounded and in this case $\sigma_\varphi\restriction_{\ell^2(\tau)}:\ell^2(\tau)\to\ell^2(\tau)$
is continuous, consequently $\sigma_\varphi\restriction_{\ell^2(\tau)}:\ell^2(\tau)\to\ell^2(\tau)$ is a compact operator
if and only if $\tau$ is finite.
\end{abstract}
\maketitle
\noindent {\small {\bf 2010 Mathematics Subject Classification:}  46C99 \\
{\bf Keywords:}} Compact operator, Generalized shift, Hilbert space.
\section{Preliminaries}
\noindent The concept of generalized shifts has been introduced for the first time 
in \cite{AHK} as a generalization of one-sided shift
$\mathop{\{1,\ldots,k\}^{\mathbb N}\to\{1,\ldots,k\}^{\mathbb N}}\limits_{
\SP(a_1,a_2,\cdots)\mapsto(a_2,a_3,\cdots)}$
and two-sided shift
$\mathop{\{1,\ldots,k\}^{\mathbb Z}\to\{1,\ldots,k\}^{\mathbb Z}}\limits_{
\SP(a_n)_{n\in\mathbb Z}\mapsto(a_{n+1})_{n\in\mathbb Z}}$~\cite{walters, shift}. 
Suppose $K$ is a nonempty set with at least two elements, $\Gamma$
is a nonempty set, and $\varphi:\Gamma\to\Gamma$ is an arbitrary map, then
we call
\linebreak
$\sigma_\varphi:\mathop{K^\Gamma\to K^\Gamma\SP}\limits_{(x_\alpha)_{\alpha\in\Gamma}
\mapsto(x_{\varphi(\alpha)})_{\alpha\in\Gamma}}$ a generalized shift
(for one-sided and two-sided shifts consider $\varphi(n)=n+1$).
It's evident that for topological space $K$, $\sigma_\varphi:K^\Gamma\to K^\Gamma$
is continuous, where $K^\Gamma$ is equipped by product topology.
\\
For Hilbert space $H$ there exists unique cardinal number $\tau$  such that
$H$ and $\ell^2(\tau)$ are isomorphic~\cite{hilbert1, hilbert2}. All members of the collection 
$\{\ell^2(\tau):\tau$ is a non--zero cardinal number$\}$ are Hilbert spaces, moreover for
cardinal number $\tau$ and $(x_\alpha)_{\alpha<\tau}\in{\mathbb K}^\tau$
(where $\mathbb{K}\in\{\mathbb{R},\mathbb{C}\}$ depending on our
choice for real Hilbert spaces or Complex Hilbert spaces)
we have  $x=(x_\alpha)_{\alpha<\tau}\in\ell^2(\tau)$ if and only if
$||x||^2:=\mathop{\Sigma}\limits_{\alpha<\tau}|x_\alpha|^2<+\infty$.
Moreover for $(x_\alpha)_{\alpha<\tau},(y_\alpha)_{\alpha<\tau}\in\ell^2(\tau)$
let $<(x_\alpha)_{\alpha<\tau},(y_\alpha)_{\alpha<\tau}>=
\mathop{\Sigma}\limits_{\alpha<\tau}x_\alpha\overline{y_\alpha}$
(inner product).
For $\varphi:\tau\to\tau$, one may consider $\sigma_\varphi:\mathbb{K}^\tau\to
\mathbb{K}^\tau$ in particular we may study $\sigma_\varphi\restriction_{\ell^2(\tau)}:\ell^2(\tau)\to
\mathbb{K}^\tau$.
\\
{\bf Convention.} In the following text suppose $\tau>1$ is a cardinal number and
$\varphi:\tau\to\tau$ is arbitrary, we denote  $\sigma_\varphi\restriction_{\ell^2(\tau)}:\ell^2(\tau)\to
\mathbb{K}^\tau$ simply by  $\sigma_\varphi:\ell^2(\tau)\to\mathbb{K}^\tau$,
and equip $\ell^2(\tau)$ with its usual inner product introduced in the above lines.
Also for cardinal number $\psi$ let (for properties of cardinal numbers and their arithmetic see~\cite{cardinal}):
\[\psi^*:=\left\{\begin{array}{lc} \psi & \psi{\rm \: is \: finite \:,} \\ +\infty & {\rm otherwise\:.}
\end{array}\right.\]
Moreover for $s\neq t$ let $\delta_s^t=0$ and $\delta_s^s=1$.
\\
If $X,Y$ are normed vector spaces, we say the linear map $S:X\to Y$ is an operator if
it is continuous. We call $(X,T)$ a linear dynamical system,
if $X$ is a normed vector space and $T:X\to X$ is an operator~\cite{linearchaos}. 
\\
Let's recall that $\mathbb{R}$ is the set of real numbers, $\mathbb{C}$ is the set of complex numbers, and $\mathbb{N}=\{1,2,\ldots\}$ is the set of natural numbers.
\section{On generalized shift operators}
\noindent In this section we show $\sigma_\varphi(\ell^2(\tau))\subseteq\ell^2(\tau)$ (and 
$\sigma_\varphi:\ell^2(\tau)\to\ell^2(\tau)$ is continuous) if and only if 
$\varphi:\tau\to\tau$ is bounded. Moreover
	$\sigma_\varphi(\ell^2(\tau))=\ell^2(\tau)$ if and only if 
	$\varphi:\tau\to\tau$ is  one--to--one.
\begin{remark}
We say $f:A\to A$ is bounded if 
there exists finite $n\geq1$ such that for all $a\in A$ 
we have ${\rm card}(\varphi^{-1}(a))\leq n$~\cite{giordano}.
\end{remark}
\begin{theorem}\label{lem10}
The following statements are equivalent:
\begin{itemize}
\item[1.] $\sigma_\varphi(\ell^2(\tau))\subseteq\ell^2(\tau)$,
\item[2.] $\varphi:\tau\to\tau$ is bounded,
\item[3.] $\sigma_\varphi:\ell^2(\tau)\to\ell^2(\tau)$ is a linear continuous map.
\end{itemize}
Moreover in the above case we have
$||\sigma_\varphi||=\sqrt{\sup\{({\rm card}(\varphi^{-1}(\alpha)))^*:\alpha\in\tau\}}$.
\end{theorem}
\begin{proof}
First note that for $x=(x_\alpha)_{\alpha<\tau}$ we have
\[||\sigma_\varphi(x)||^2=\mathop{\Sigma}\limits_{\alpha<\tau}|x_{\varphi(\alpha)}|^2
=\mathop{\Sigma}\limits_{\alpha<\tau}\left(({\rm card}(\varphi^{-1}(\alpha)))^*|x_\alpha|^2\right)\tag{*}\]
(where $0(+\infty)=(+\infty)0=0$).
\\
``(1) $\Rightarrow$ (2)'' Suppose $\sigma_\varphi(\ell^2(\tau))\subseteq\ell^2(\tau)$, for
$\theta<\tau$ we have
$||(\delta_\alpha^\theta)_{\alpha<\tau}||=1$ and:
\[||\sigma_\varphi((\delta_\alpha^\theta)_{\alpha<\tau})||^2=({\rm card}(\varphi^{-1}(\theta)))^*\]
by (*). Hence 
$(\delta_\alpha^\theta)_{\alpha<\tau}\in\ell^2(\tau)$ and
$\sigma_\varphi((\delta_\alpha^\theta)_{\alpha<\tau})\in\sigma_\varphi(\ell^2(\tau))\subseteq\ell^2(\tau)$,
thus $\varphi^{-1}(\theta)$ is finite. 
\\
Thus $\{{\rm card}(\varphi^{-1}(\alpha)):\alpha\in\tau\}$ is a collection of finite cardinal numbers. If 
\linebreak
$\sup\{({\rm card}(\varphi^{-1}(\alpha)))^*:\alpha\in\tau\}=+\infty$, then there exists
a strictly increasing sequence $\{n_k\}_{k\geq1}$ in $\mathbb N$ and sequence $\{\alpha_k\}_{k\geq1}$
in $\tau$ such that for all $k\geq1$ we have ${\rm card}(\varphi^{-1}(\alpha_k))=n_k$. Since
$\{n_k\}_{k\geq1}$ is a one--to--one sequence, $\{\alpha_k\}_{k\geq1}$ is a one--to--one sequence
too.
Consider  $(x_\alpha)_{\alpha<\tau}$ with:
\[x_\alpha:=\left\{\begin{array}{lc} \frac{1}{k} & \alpha=\alpha_k,k\geq1\:, \\ 0 & {\rm otherwise}\:.
\end{array}\right.\]
Then $\mathop{\Sigma}\limits_{\alpha<\tau}|x_\alpha|^2=
\mathop{\Sigma}\limits_{k\geq1} x_{\alpha_k}^2=\mathop{\Sigma}\limits_{k\geq1}\frac{1}{k^2}<+\infty$
and $(x_\alpha)_{\alpha<\tau}\in\ell^2(\tau)$. On the other hand by (*) we have
$||\sigma_\varphi((x_\alpha)_{\alpha<\tau})||^2=
\mathop{\Sigma}\limits_{k\geq1}\frac{n_k}{k^2}\geq\mathop{\Sigma}\limits_{k\geq1}\frac{1}{k}=+\infty$
(note that $n_k\geq k$ for all $k\geq1$), in particular
$\sigma_\varphi((x_\alpha)_{\alpha<\tau})\notin\ell^2(\tau)$
which leads to the contradiction 
$\sigma_\varphi(\ell^2(\tau))\not\subseteq\ell^2(\tau)$. Therefore
$\sup\{{\rm card}(\varphi^{-1}(\alpha)):\alpha\in\tau\}$ is finite and is a natural number.
\\
``(2) $\Rightarrow$ (3)'' Suppose 
$n:=\sup\{{\rm card}(\varphi^{-1}(\alpha)):\alpha\in\tau\}$ is finite.
For all $x=(x_\alpha)_{\alpha<\tau}\in\ell^2(\tau)$
we have:
\[||\sigma_\varphi(x)||
=\sqrt{\mathop{\Sigma}\limits_{\alpha<\tau}\left(({\rm card}(\varphi^{-1}(\alpha)))^*|x_\alpha|^2\right)}
\leq\sqrt{\mathop{\Sigma}\limits_{\alpha<\tau}\left(n|x_\alpha|^2\right)}=\sqrt{n}||x||\]
which shows continuity of $\sigma_\varphi$ and $||\sigma_\varphi||\leq\sqrt{n}$.
On the other hand, there exists $\theta<\tau$ with
${\rm card}(\varphi^{-1}(\theta)=n$. By $||(\delta_\alpha^\theta)_{\alpha<\tau}||=1$ and (*) we have
$||\sigma_\varphi((\delta_\alpha^\theta)_{\alpha<\tau})||=\sqrt{n}$ which leads to
$||\sigma_\varphi||\geq\sqrt{n}$.
\end{proof}
\noindent By \cite[Lemma 4.1]{set} and \cite{AHK}, $\varphi:\tau\to\tau$ is one--to--one (resp. onto) if and only if 
$\sigma_\varphi:{\mathbb K}^\tau\to{\mathbb K}^\tau$
is onto (resp. one--to--one), however the following lemma deal with 
$\sigma_\varphi:\ell^2(\tau)\to\ell^2(\tau)$.
\begin{lemma}\label{lem20}
The following statements are equivalent:
\begin{itemize}
\item[1.] $\sigma_\varphi(\ell^2(\tau))=\ell^2(\tau)$,
\item[2.] $\sigma_\varphi(\ell^2(\tau))$ is a dense subset of $\ell^2(\tau)$,
\item[3.] $\varphi:\tau\to\tau$ is one--to--one.
\end{itemize}
In addition the following statements are equivalent too:
\begin{itemize}
\item[i.] $\sigma_\varphi:\ell^2(\tau)\to{\mathbb K}^\tau$ is 
one--to--one,
\item[ii.] $\varphi:\tau\to\tau$ is onto.
\end{itemize}
\end{lemma}
\begin{proof}
``(2) $\Rightarrow$ (3)'' Suppose $\varphi:\tau\to\tau$ 
is not one--to--one, then there exists
$\theta\neq\psi$ with $\mu:=\varphi(\theta)=\varphi(\psi)$. 
There exists
$(y_\alpha)_{\alpha<\tau}\in\ell^2(\tau)$ with $||\sigma_\varphi((y_\alpha)_{\alpha<\tau})-
(\delta_\alpha^\theta)_{\alpha<\tau}||<\frac14$, thus for all $\alpha<\tau$ we have 
$|y_{\varphi(\alpha)}-\delta_\alpha^\theta|<\frac14$ in particular 
$|y_{\varphi(\psi)}-\delta_\psi^\theta|<\frac14$ and $|y_{\varphi(\theta)}-\delta_\theta^\theta|<\frac14$,
thus $|y_{\mu}|<\frac14$ and $|y_{\mu}-1|<\frac14$, which is 
a contradiction, therefore $\varphi:\tau\to\tau$ is one--to--one.
\\
``(3) $\Rightarrow$ (1)'' Suppose $\varphi:\tau\to\tau$ 
is one--to--one, then by
Theorem~\ref{lem10}, $\sigma_\varphi(\ell^2(\tau))\subseteq\ell^2(\tau)$. For
$y=(y_\alpha)_{\alpha<\tau}\in\ell^2(\tau)$, define $x=(x_\alpha)_{\alpha<\tau}$
in the following way:
\[x_\alpha=\left\{\begin{array}{lc} y_\beta & \alpha=\varphi(\beta),\beta<\tau\:, \\
0 & \alpha\in\tau\setminus\varphi(\tau)\:,\end{array}\right.\]
then $||x||=||y||$ and $x\in\ell^2(\tau)$,  moreover $\sigma_\varphi(x)=y$, which leads to
$\sigma_\varphi(\ell^2(\tau))=\ell^2(\tau)$.
\\
In order to complete the proof we should prove that (i) and (ii) are equivalent
however by \cite[Lemma 4.1]{set}, (ii) implies (i), so we should just prove that (i) implies (ii).
\\
``(i) $\Rightarrow$ (ii)'' Suppose $\varphi:\tau\to\tau$ is not onto and choose $\theta\in\tau\setminus
\varphi(\tau)$. 
Then $(\delta_\alpha^\theta)_{\alpha<\tau},(0)_{\alpha<\tau}$ are two distinct elements of
$\ell^2(\tau)$, however \[\sigma_\varphi((\delta_\alpha^\theta)_{\alpha<\tau})=
\sigma_\varphi((0)_{\alpha<\tau})=(0)_{\alpha<\tau}\]
and 
$\sigma_\varphi:\ell^2(\tau)\to{\mathbb K}^\tau$ is not one--to--one.
\end{proof}
\begin{corollary}
The following statements are equivalent:
\begin{itemize}
\item[1.] $\varphi:\tau\to\tau$ is bijective,
\item[2.] $\sigma_\varphi:\ell^2(\tau)\to\ell^2(\tau)$ is bijective,
\item[3.] $\sigma_\varphi:\ell^2(\tau)\to\ell^2(\tau)$ is an isomorphism,
\item[4.] $\sigma_\varphi:\ell^2(\tau)\to\ell^2(\tau)$ is an isometry.
\end{itemize}
\end{corollary}
\begin{proof}
Using Lemma~\ref{lem20}, (1) and (2) are equivalent. It's evident that (3) implies
(2), moreover if $\varphi:\tau\to\tau$ is bijective, then by Theorem~\ref{lem10} two linear
maps $\sigma_\varphi:\ell^2(\tau)\to\ell^2(\tau)$ and its inverse 
$\sigma_{\varphi^{-1}}:\ell^2(\tau)\to\ell^2(\tau)$ are continuous, hence (1) implies (3).
\\
(1) implies (4),
is evident by (*) in Theorem~\ref{lem10}. 
In order to complete the proof, we should just prove that (4) implies (1).
\\
``(4) $\Rightarrow$ (1)'' Suppose $\sigma_\varphi:\ell^2(\tau)\to\ell^2(\tau)$ is an isometry,
then $\sigma_\varphi:\ell^2(\tau)\to\ell^2(\tau)$ is one--to--one and by Lemma~\ref{lem20},
$\varphi:\tau\to\tau$ is onto. Moreover, $||\sigma_\varphi||=1$ since 
$\sigma_\varphi:\ell^2(\tau)\to\ell^2(\tau)$ is an isometry. By Lemma~\ref{lem10}
we have $1=||\sigma_\varphi||^2=\sup\{{\rm card}(\varphi^{-1}(\alpha)):\alpha\in\tau\}$,
thus for all $\alpha<\tau$ we have ${\rm card}(\varphi^{-1}(\alpha))\leq1$ and
$\varphi:\tau\to\tau$ is one--to--one.
\end{proof}
\begin{lemma}\label{th}
Let $\mathcal{D}=\{z\in\ell^2(\tau):\sigma_\varphi(z)\in\ell^2(\tau)\}$
(consider $\mathcal D$ with induced normed and topology of
$\ell^2(\tau)$), then:
\begin{itemize}
\item[1.] $\mathcal{D}$ is a subspace of $\ell^2(\tau)$,
\item[2.] $\{\theta<\tau:\exists(z_\alpha)_{\alpha<\tau}\in \mathcal{D}\: z_\theta\neq0\}=\{\alpha<\tau:\varphi^{-1}(\alpha)$ is finite $\}$.
\end{itemize}
\end{lemma}
\begin{proof}
Since $\sigma_\varphi:\ell^2(\tau)\to{\mathbb K}^\tau$ is linear, we have immediately (1).
\\
2) We have 
\[||\sigma_\varphi((\delta_\alpha^\theta)_{\alpha<\tau})||=\left\{\begin{array}{lc} \sqrt{({\rm card}(\varphi^{-1}(\theta)))^*} &
\varphi^{-1}(\theta){\rm \: is \: finite,} \\ +\infty & \varphi^{-1}(\theta) {\rm \: is \: infinite.} \end{array}\right.\tag{**}\]
Thus if $\varphi^{-1}(\theta)$ is finite we have $\sigma_\varphi((\delta_\alpha^\theta)_{\alpha<\tau})\in\ell^2(\tau)$
and $(\delta_\alpha^\theta)_{\alpha<\tau}\in{\mathcal D}$, which shows 
$\theta\in\{\beta<\tau:\exists(z_\alpha)_{\alpha<\tau}\in \mathcal{D}\: z_\beta\neq0\}$. Therefore:
\begin{center}
$\{\alpha<\tau:\varphi^{-1}(\alpha)$ is finite $\}\subseteq \{\theta<\tau:\exists(z_\alpha)_{\alpha<\tau}\in \mathcal{D}\: z_\theta\neq0\}$
\end{center}
Now for $\theta<\tau$ suppose there exists $(z_\alpha)_{\alpha<\tau}\in \mathcal{D}$ with $z_\theta\neq0$.
Using the fact that $\mathcal D$ is a subspace of 
$\ell^2(\tau)$ we may suppose $z_\theta=1$, now we have
(since $\sigma_\varphi(z)\in\ell^2(\tau)$): 
\[||\sigma_\varphi((\delta^\theta_\alpha)_{\alpha<\tau})||=||\sigma_\varphi((z_\alpha \delta^\theta_\alpha)_{\alpha<\tau})||\leq ||\sigma_\varphi(z)||<+\infty\:,\]
by (**), $\varphi^{-1}(\theta)$ is finite, which completes the proof of (2).
\end{proof}
\begin{note}\label{note}
For $H\subseteq\tau$ the closure of subspace generated
by $\{(\delta_\alpha^\theta)_{\alpha<\tau}:\theta\in H\}$ (in $\ell^2(\tau)$)
is $\{(x_\alpha)_{\alpha<\tau}\in\ell^2(\tau):\forall\alpha\notin H\:(x_\alpha=0)\}$.
\end{note}
\begin{theorem}
For $\mathcal{D}=\{z\in\ell^2(\tau):\sigma_\varphi(z)\in\ell^2(\tau)\}$ as in Lemma~\ref{th} and 
$M:=\{\alpha<\tau:\varphi^{-1}(\alpha)$ is finite $\}$, the following statemnts are equivalent:
\begin{itemize}
\item[1.] $\sigma_\varphi\restriction_{\mathcal D}:\mathcal{D}\to\ell^2(\tau)$ is continuous,
\item[2.] there exists finite $n\geq1$ with ${\rm card}(\varphi^{-1}(\alpha))\leq n$  for all $\alpha\in M$,
\item[3.] $\mathcal{D}=\{(x_\alpha)_{\alpha<\tau}\in\ell^2(\tau):\forall\theta\notin M\:\: x_\theta=0\}$,
\item[4.] $\mathcal{D}$ is a closed subspace of $\ell^2(\tau)$,
\end{itemize}
\end{theorem}
\begin{proof} ``(1) $\Rightarrow$ (2)''
Suppose $\sigma_\varphi\restriction_{\mathcal D}:\mathcal{D}\to\ell^2(\tau)$ is continuous,
consider $\theta<\tau$ with finite $\varphi^{-1}(\theta)$. 
By proof of item (2) in Lemma~\ref{th}, 
$(\delta_\alpha^\theta)_{\alpha<\tau}\in{\mathcal D}$ and 
$||\sigma_\varphi((\delta_\alpha^\theta)_{\alpha<\tau})||=\sqrt{({\rm card}(\varphi^{-1}(\theta)))^*}$,
thus 
\[+\infty>||\sigma_\varphi||\geq\sup\{\sqrt{({\rm card}(\varphi^{-1}(\theta)))^*}:\theta\in M\}\:.\]
``(2) $\Rightarrow$ (1)''
For
$n:=\sup\{{\rm card}(\varphi^{-1}(\alpha)):\alpha\in M\}<+\infty$
and $x=(x_\alpha)_{\alpha<\tau}\in{\mathcal D}$
we have $||\sigma_\varphi(x)||
=\sqrt{\mathop{\Sigma}\limits_{\alpha\in M}\left(({\rm card}(\varphi^{-1}(\alpha)))^*|x_\alpha|^2\right)}
\leq\sqrt{\mathop{\Sigma}\limits_{\alpha\in M}\left(n|x_\alpha|^2\right)}=\sqrt{n}||x||$
which shows continuity of $\sigma_\varphi\restriction_{\mathcal D}:\mathcal{D}\to\ell^2(\tau)$.
\\
``(3) $\Rightarrow$ (4)''
Note that for nonempty $M$, $\{(x_\alpha)_{\alpha<\tau}\in\ell^2(\tau):\forall\theta\notin M\:\: x_\theta=0\}$
is just $\ell^2(M)$.
\\
``(4) $\Rightarrow$ (3)''
By proof and notations in Lemma~\ref{th}, we have 
$\{(\delta_\alpha^\theta)_{\alpha<\tau}:\theta\in M\}\subseteq \mathcal D$. Use Note~\ref{note} to complete the proof.
\\
``(2) $\Rightarrow$ (3)''
For
$n:=\sup\{{\rm card}(\varphi^{-1}(\alpha)):\alpha\in M\}<+\infty$
and $x=(x_\alpha)_{\alpha<\tau}\in\ell^2(\tau)$ with $x_\alpha=0$ for all $\alpha\notin M$ 
we have $||\sigma_\varphi(x)||\leq\sqrt{n}||x||$ which shows $x\in{\mathcal D}$ and 
$\{(x_\alpha)_{\alpha<\tau}\in\ell^2(\tau):\forall\theta\notin M\:\: x_\theta=0\}\subseteq\mathcal D$.
Using Lemma~\ref{th} we have
$\mathcal{D}\subseteq\{(x_\alpha)_{\alpha<\tau}\in\ell^2(\tau):\forall\theta\notin M\:\: x_\theta=0\}$.
\\
``(3) $\Rightarrow$ (2)''
If $\sup\{({\rm card}(\varphi^{-1}(\alpha)))^*:\alpha\in M\}=+\infty$, then there exists
a strictly increasing sequence $\{n_k\}_{k\geq1}$ in $\mathbb N$ and sequence $\{\alpha_k\}_{k\geq1}$
in $M$ such that for all $k\geq1$ we have ${\rm card}(\varphi^{-1}(\alpha_k))=n_k$. Using a similar method described in Lemma~\ref{lem10},
consider  $(x_\alpha)_{\alpha<\tau}\in{\mathcal D}$ with:
\[x_\alpha:=\left\{\begin{array}{lc} \frac{1}{k} & \alpha=\alpha_k,k\geq1\:, \\ 0 & {\rm otherwise}\:.
\end{array}\right.\]
Then 
$||\sigma_\varphi((x_\alpha)_{\alpha<\tau})||^2=
\mathop{\Sigma}\limits_{k\geq1}\frac{n_k}{k^2}\geq\mathop{\Sigma}\limits_{k\geq1}\frac{1}{k}=+\infty$, in particular
$\sigma_\varphi((x_\alpha)_{\alpha<\tau})\notin\ell^2(\tau)$
which is in contradiction with
$(x_\alpha)_{\alpha<\tau}\in{\mathcal D}$ and completes the proof.
\end{proof}
\subsection{Compact generalized shift operators}
For normed vector spaces $X,Y$ we say the operator $T:X\to Y$ is a
compact operator, if $\overline{T(B^X(0,1))}$ is a compact subset of $Y$, 
where $B^X(0,1)=\{x\in X:||x||<1\}$~\cite{conway}.
\begin{theorem}
The generalized shift operator $\sigma_\varphi:\ell^2(\tau)\to\ell^2(\tau)$ is compact if and only if $\tau$ is finite.
\end{theorem}
\begin{proof}
If $\tau$ is infinite, then by Theorem~\ref{lem10}, $\{\varphi^{-1}(\alpha):\alpha<\tau\}\setminus\{\varnothing\}$
is a partition of $\tau$ to its finite subsets, thus there exists a one--to--one sequence $\{\alpha_n\}_{n\geq1}$
in $\tau$ such that $\{\varphi^{-1}(\alpha_n\}_{n\geq1}$ is a sequence of nonempty finite
and disjoint subsets of $\tau$. For all distinct $n,m\geq1$ we have 
\[||\sigma_\varphi((\delta_\alpha^{\alpha_n})_{\alpha<\tau})-
\sigma_\varphi((\delta_\alpha^{\alpha_m})_{\alpha<\tau})||=\sqrt{2}\]
so $\{\sigma_\varphi(\frac12(\delta_\alpha^{\alpha_n})_{\alpha<\tau})\}_{n\geq1}$ (is a sequence in
$\overline{\sigma_\varphi(B((0)_{\alpha<\tau},1))}$) without any converging subsequence.
Therefore $\overline{\sigma_\varphi(B((0)_{\alpha<\tau},1))}$ is not compact and 
$\sigma_\varphi:\ell^2(\tau)\to\ell^2(\tau)$ is not a compact operator.
\\
On the other hand, if $\tau$ is finite, then every linear operator on $\ell^2(\tau)$
is compact, hence $\sigma_\varphi:\ell^2(\tau)\to\ell^2(\tau)$ is a compact operator.
\end{proof}

\noindent {\scriptsize {\bf Fatemah Ayatollah Zadeh Shirazi},
Faculty of Mathematics, Statistics and Computer Science,
College of Science, University of Tehran ,
Enghelab Ave., Tehran, Iran
({\it e-mail}: fatemah@khayam.ut.ac.ir)
\\
{\bf Fatemeh Ebrahimifar},
Faculty of Mathematics, Statistics and Computer Science,
College of Science, University of Tehran ,
Enghelab Ave., Tehran, Iran
({\it e-mail}: ebrahimifar64@ut.ac.ir)}
\end{document}